%% file: ElekesSzabo.tex
\documentclass[12pt]{amsart}
\usepackage{amsmath,amssymb,amsfonts}
\usepackage{mathtools}
\usepackage[shortlabels]{enumitem}
\setlist{leftmargin=*}
\usepackage{tikz}
\usetikzlibrary{arrows}
\usetikzlibrary{decorations.markings}

\usepackage[colorlinks=true, linkcolor=blue]{hyperref}

\usepackage[sorted]{amsrefs}
\usepackage[textsize=tiny]{todonotes}

\include{declar}
\include{localdef}
\title[Model-theoretic Elekes-Szab\'o in the strongly minimal case]{Model-theoretic Elekes-Szab\'o in the strongly minimal case}

\author{Artem Chernikov} 
\address{Department of Mathematics\\
University of California Los Angeles\\
Los Angeles, CA 90095-1555} 
\email{chernikov@math.ucla.edu}

\author{Sergei Starchenko}
\address{Department of Mathematics\\
University of Notre Dame\\
Notre Dame, IN, 46656, USA}
\email{sstarche@nd.edu}

\begin{document}
\maketitle

 \gdef\tY{\tilde{Y}}
\gdef\tZ{\tilde{Z}}
\gdef\RM{\operatorname{RM}}
\gdef\acl{\operatorname{acl}}
\global\long\def\cdim{\operatorname{cdim}}
\gdef\Im{\operatorname{Im}}

\begin{abstract}
We prove a generalizations of the Elekes--Szab\'{o} theorem \cite{ES} for
relations  definable in 
strongly minimal structures that are interpretable in distal
structures.
\end{abstract}

\section{Introduction and preliminaries}
\label{sec:setting}

\bigskip 
Our notation is mostly standard. For $n\in \NN$ we denote by $[n]$ the set $[n]=\{1,\dotsc,n\}$. If $X$ is a set and $n\in \NN$, 
then we write $A \subseteq_n X$ to denote that $A$ is a subset of $X$
with $|A|\leq n$. Given a binary relation $E \subseteq X \times Y$ and $a \in X, b \in Y$, we write $E_a = \{ b' \in Y : (a,b') \in E \}$ and $E_b = \{ a' \in X : (a',b) \in Y\}$ to denote the fibers of $E$ at $a$ and $b$, respectively.
 As usual, for  functions $f,g\colon \NN\to \RR$ we write 
 $f(n)  = O(g(n))$ if there is a positive $C\in \RR$ and $n_0\in
\NN$ such that 
$f(n) \leq C g(n)$ for all $n> n_0$.
As usual, given $s,t \in \mathbb{N}$ we say that a bipartite graph $E \subseteq X \times Y$ is \emph{$K_{s,t}$-free} if it does not contain a copy of the complete bipartite graph $K_{s,t}$ with parts of size $s$ and $t$, respectively.

We will use freely some standard model-theoretic notions such as saturated models, algebraic closure (by which we will always mean the algebraic closure in $\CM^\eq$) and Morley rank (see e.g. \cite{marker2006model, tent2012course}).

\begin{defn}
  \begin{enumerate}
  \item 
We say that a subset $F\subseteq X\times Y$ is \emph{cartesian}  if
$I\times J \subseteq F$ for some infinite $I\subseteq X, J\subseteq
Y$. 
\item  We say that a subset $F\subseteq S_1\times S_2\times \dotsb \times  S_k$ is \emph{cylindrical}  if
  it is cartesian as a subset of  $S_i\times \hat{S_i}$ for some $i\in
  [k]$, where $\hat{S_i}=\prod_{j\neq i} S_j$.
\end{enumerate}
\end{defn}

Let $\CM$ be a sufficiently saturated first order structure and let $X,Y,Z$ be strongly minimal
sets definable in $\CM$.   Let $F\subseteq X\times Y\times Z$ be a definable set of
Morley rank $2$. As usual, we say that $\bar a=(a_1,a_2,a_3)\in F$ is
\emph{generic in $F$}  over a set of parameters $C\subseteq   X\times Y\times Z$ if
$\RM(\bar a/C)=\RM(F)=2$. 

\begin{defn} \label{def: group-like}
 We say that a definable relation $F$ as above is \emph{group-like} if there is a group $G$ of Morley
rank $1$ and degree $1$ (hence, abelian) definable in $\CM$ over
a small set $C \subseteq X \times Y \times Z$,  elements
$g_1,g_2,g_3\in G$, and $\alpha_1\in X$, $\alpha_2\in Y$, $\alpha_3\in Z$
such that $\alpha_i$ and $g_i$ are inter-algebraic over $C$ for all $i \in [3]$ (i.e.~$\alpha_i \in \acl(g_i C)$ and $g_i \in \acl(a_i C)$), $\bar
\alpha=(\alpha_1,\alpha_2,\alpha_3)\in F$ is generic in $F$ over $C$
and $g_1\cdot g_2\cdot g_3=1$ in $G$.
\end{defn}

\medskip

We can now state our main result.

\begin{thm}[Main Theorem]\label{thm:main0} 
Let $X,Y,Z$ be strongly minimal sets definable in a sufficiently saturated structure  $\CM$ and let $F \subseteq X\times Y\times Z$ be a
definable set of Morley rank $2$. 
Assume in addition that $\CM$ is interpretable in a distal
structure (see Section \ref{sec: distal Zarank}). Then one of the following holds. 
\begin{enumerate}[(a)]
\item  There is some real $\varepsilon>0$ such that for all $A\subseteq_n X ,B
  \subseteq_n Y ,C \subseteq_n Z$
  we have 
\[|F\cap A\times B \times C | = O(n^{2-\varepsilon}). \] 
\item  $F$ is group-like.
\item $F$ is cylindrical.
\end{enumerate}
\end{thm}

\begin{rem}
Theorem \ref{thm:main0}  can be viewed as a generalization of the 
Elekes-Szab\'{o} Theorem \cite{ES} which established it for $\CM = \mathbb{C}$ the field of complex numbers (a strongly minimal structure), which is interpretable in the field of reals --- a distal structure.

\end{rem}

\begin{rem}
Various improvements of the Elekes-Szab\'{o} theorem, including explicit bounds on $\varepsilon$,  have been obtained \cite{raz2016polynomials, deZ, wang2015exposition}.
In our general situation we don't optimize the bounds, even though they can be calculated explicitly in terms of the size of the available distal cell decomposition, see Section \ref{sec: distal Zarank}.
\end{rem}

\begin{rem}
After the completion of this paper, generalizations to definable hypergraphs of arbitrary arity and dimension were obtained in \cite{bays2018projective} over $\mathbb{C}$, in \cite{raz2018expanding} for a higher arity version of the Elekes-R\'onyai theorem over $\mathbb{C}$, and in \cite{ChePeSt} for hypergraphs of arbitrary arity and dimension definable in arbitrary stable structures with distal expansions, as well as in arbitrary $o$-minimal structures.
\end{rem}

%We can now state the Elekes-Szab\'{o} Theorem.
%
%\begin{thm}[Elekes-Szab\'{o} Theorem, \cite{ES}]
%Let $F\subseteq \CC^3$ be a definable (in the field language) set of
%Morley Rank two.  Then one of the following holds.
%\begin{enumerate}[(a)]
%\item  There is $\varepsilon>0$ such that for  $A,B,C \subseteq_n \CC$
%  we have 
%\[|F\cap A\times B \times C | = O(n^{2-\varepsilon}). \] 
%\item  $F$ is group-like
%\item $F$ is cylindrical.
%\end{enumerate}
%\end{thm}

\begin{prop}If the Morley degree of $F$ is $1$, then  the three cases in Theorem \ref{thm:main0} are mutually exclusive.	
\end{prop}

\begin{proof}
\emph{(a) and (c) are incompatible.} Assume $F$ is cylindrical, say $F \subseteq (X\times Y) \times
  Z$ is cartesian. Then for any $n$ there are some $D \subseteq
  X\times Y, C \subseteq Z$ with $|D|=|C|=n$ such that $D \times C \subseteq F$. Let $A ,B$ be the projections of $D$ on $X$ and $Y$, respectively. Then $|A|, |B| \leq n$ and $|F \cap A \times B \times C|\geq n^2$.
  
\emph{(b) and (c) are incompatible.} By assumption $F$ has Morley rank $2$ and degree $1$, so it contains a unique generic type. We work over some saturated model. Assume that $F$ is cylindrical, say $F \subseteq (X \times Y) \times Z$ is cartesian. Then there exist $\bar{a} = (a_1,a_2) \in X \times Y$ and $b \in Z$ such that $(a_1, a_2, b) \in F, \bar{a} \notin \acl(b)$ and $b \notin \acl(\bar{a})$. Since $Z$ is strongly minimal, it follows that $\bar{a} \ind b$ and $(\bar{a},b)$ is generic in $F$. So a generic of $F$ has to be of the form $(\alpha, \beta, \gamma)$ with $\RM(\alpha \beta) = \RM(\gamma) = 1$ and $\alpha \beta \ind \gamma$. On the other hand, if $F$ is group-like, then for  its generic $(\alpha, \beta, \gamma)$ we must have $\RM(\alpha \beta) = 2$ and $\gamma \in \acl(\alpha,\beta)$.
  
 \emph{(a) and (b) are incompatible.}
 Let $G$ be a definable abelian group with $\RM(G) = 1$, and let $F_G \subseteq G^3$ be the (non-definable) relation given by 
$$(a, b, c) \in F_G :\iff a \ind b \textrm{ are generic in }G \textrm{ and } a + b + c = 0.$$

First we establish high edge count for the relation $F_G$.

\begin{claim}\label{cla: high count on groups}
	There exists some $c = c(G) \in \mathbb{R}_{>0}$ such that for arbitrary large $n \in \mathbb{N}$ there exist some $A,B,C \subseteq_n G$ such that:
 $$|\{F_G \cap (A \times B \times C) \}| \geq c n^2.$$
\end{claim}
\begin{proof}[Proof of Claim \ref{cla: high count on groups}]

Assume first that $G$ has a generic type $p\left(x\right)$ such
that some/any realization of $p$ is of infinite order. Choose some
$a,b\models p$ with $a\ind b$, then $nb\in\acl\left(b\right)$ and
$b\in\acl\left(nb\right)$, hence $a\ind nb$ and $nb$ is generic, 
for all $n\in\mathbb{N}$. Fix $n\in\mathbb{N}$, and consider the
set $A=\left\{ a,a+b,a+2b,\ldots,a+(n-1)b\right\} $, then $|A| = n$. By forking calculus, the elements of $A$
are all generic in $G$ and pairwise independent: $a+sb\ind a+tb$ for $s\neq t$.
Note that for any $s,t < \frac{n}{2}$ we have $(a+sb) + (a+tb) = a + (s+t)b \in A$ as $s+t < n$, hence taking $B := A$ and $C:= -A$, we see that 
$$|F_G \cap (A \times B \times C)| \geq \left|\left\{(s,t): s\neq t \textrm{ and } s,t < \frac{n}{2} \right\} \right| \geq  \left(\frac{n}{2} \right)^2 - \frac{n}{2} \geq \frac{1}{6} n^2,$$
for all sufficiently large $n$, as wanted.

Otherwise, all generics of $G$ have the same finite order, say $k$, and as every elements in a stable group is a product of two generics, by commutativity this implies that $kx=0$ for all $x \in G$.
Then, by classification of abelian groups, $px=0$ holds in $G$ for some prime $p$, and  $G$ is an $\mathbb{F}_{p}$-vector space.
Fix any $m \in \mathbb{N}$ and choose independent generic elements $a_{1},\ldots,a_{m}$ in $G$.
Let $A:=\left\langle a_{1},\ldots,a_{m}\right\rangle $ be the $\mathbb{F}_p$-linear span of $\{a_1, \ldots, a_m \}$, then $\left|A\right|=p^{m} := n$.
By forking calculus we have: every $0 \neq a \in A$ is generic in $G$, and for any $a,b \in A$, $a \ind b$ if and only if $a$ and $b$ are linearly independent over $\mathbb{F}_p$.
But then we have
$$|F_G \cap A^3| \geq |\{(a,b) \in A^2 : a \textrm{ and } b \textrm{ are } \mathbb{F}_p\textrm{-linearly independent} \}|$$
$$= (p^m -1)(p^m -p) \geq \frac{1}{4}n^2$$
for all $n \gg p$.
\end{proof}

Now assume that $F$ is group-like, witnessed by $G$ and $\alpha_i, g_i$ (we suppress the parameter set to simplify the notation). We want to transfer high edge count from $F_G$ to $F$. As $\alpha_i$ is inter-algebraic with $g_i$, we can choose some formulas $\varphi_i$ and $k \in \mathbb{N}_{\geq 1}$ such that 
\begin{enumerate}
	\item[$(\dagger)$] $\models \varphi_i(\alpha_i, g_i)$, and $|\varphi_i(\alpha,M)|, |\varphi_i(M,g)| \leq k$ for all $i \in \{1,2,3\}$ and $\alpha \in X \cup Y \cup Z, g \in G$.
\end{enumerate}
As $G$ has Morley rank and degree $1$, there is a unique generic type, hence by stationarity we have the equality of types $g_1 g_2 \equiv g'_1 g'_2$ for any generic $g_1 \ind g_2, g'_1 \ind g'_2$ in $G$. As $\models F(\alpha_1, \alpha_2, \alpha_3)$ by assumption, taking an automorphism we see that:
\begin{enumerate}
	\item[$(\dagger \dagger)$] for every $(g_1, g_2, g_3) \in F_G$ there exist some $\alpha_1 \in \varphi_1(X,g_1), \alpha_2 \in \varphi_2(Y, g_2), \alpha_3 \in \varphi_3(Z,g_3)$ such that $(\alpha_1, \alpha_2, \alpha_3) \in F$.
\end{enumerate} 
Now let $c \in \mathbb{R}_{>0}$ be as given by Claim \ref{cla: high count on groups} for $G$. Let $n \in \mathbb{N}$ be arbitrarily large, and let $A,B,C \subseteq_n G$ be such that $|F_G \cap (A \times B \times C)| \geq c n^2$. We define $A' := \bigcup_{g \in A} \varphi(X,g), B' := \bigcup_{g \in B} \varphi(Y,g)$ and $C' := \bigcup_{g \in C} \varphi(Z,g)$. Then $|A'|,|B'|, |C'| \leq kn := n'$ by $(\dagger)$. 
We define a function 
$$f: F_G \cap (A \times B \times C) \to F \cap (A' \times B' \times C'),$$
 where for $\bar{g} = (g_1, g_2, g_3) \in F_G \cap (A \times B \times C)$  we take $f(\bar{g})$ to be an arbitrary element $\bar{\alpha} = (\alpha_1, \alpha_2, \alpha_3) \in F \cap \left(\varphi_1(X,g_1) \times \varphi_2(Y, g_2)\times \varphi_3(Z,g_3) \right)$, it exists by $(\dagger \dagger)$. But   as $|\varphi_i(\alpha_i, G)| \leq k$ for any $(\alpha_1, \alpha_2, \alpha_3) \in X \times Y \times Z$ by $(\dagger)$, we have that $|f^{-1}(\{ \bar{\alpha} \})| \leq k^3$ for every $\bar{\alpha} \in \Im(f)$.  
Hence 
$$|F \cap (A' \times B' \times C')| \geq \frac{1}{k^3}|F_G \cap (A \times B \times C)| \geq \frac{c}{k^3} n^2 \geq \frac{c}{k^5}(n')^2 \geq c' (n')^2$$
with $c' := \frac{c}{k^5} >0$. As $n'$ can be taken arbitrarily large, this shows that $F$ doesn't satisfy (a).
\end{proof}

The proof of Theorem~\ref{thm:main0} consists of three main ingredients: a bound on the number of edges for non-cartesian relations in our context (i.e. Theorem \ref{thm:zippel} established in Section \ref{sec: cartesian relations} using local stability and the distal cutting lemma from \cite{chernikov2016cutting}); Hrushovski's group configuration theorem in stable theories; and the construction of the group configuration in the cartesian case connecting the two aforementioned results. In Section \ref{sec:main-theorem} this last part is reduced to a certain dichotomy for binary relations between sets of rank $2$, and this dichotomy is proved in Section \ref{sec:proof-theor-refthm:m}).

\subsection*{Acknowledgements}
We thank the referee for some very insightful suggestions on improving the paper.
Chernikov was supported by the NSF Research Grant DMS-1600796, by the NSF CAREER grant DMS-1651321 and by an Alfred P. Sloan Fellowship.

Starchenko was supported by the NSF Research Grant DMS-1800806.

\section{Bounds for non-cartesian relations}\label{sec: cartesian relations}

\subsection{Zarankiewicz for distal relations}\label{sec: distal Zarank}
As is well-known, if $E \subseteq A \times B$ is a bipartite graph, with $A,B$ finite, such that no two distinct points in $A$ have more than 
$s$ common neighbours in $B$ (i.e.~$E$ is $K_{2,s}$-free), then a simple argument via the Cauchy-Schwarz inequality implies that the number of edges of $E$ is at most $O_s(|A||B|^{\frac{1}{2}} + |B|)$	
 --- and this is optimal in general. The classical theorem of Szemer\'edi--Trotter \cite{szemeredi1983extremal} and its generalizations improve on this inequality, reducing the power by some $\varepsilon>0$, in the situation when $A,B$ are points in a Euclidean space and the graph $E$ is given by some (semi-)algebraic relation (we refer to \cite{shefferincidence} for a general introduction to the area of incidence geometry). In particular, a higher-dimensional ``point -- algebraic variety'' incidence bound due to Elekes and Szab\'o \cite[Theorem 9]{ES} was crucial in their proof of the $\varepsilon$-gap in the exponent in Theorem \ref{thm:main0}(a) for $\mathcal{M} = \mathbb{C}$.

In this section we generalize (and strengthen) the incidence bound due to Elekes and Szab\'o in \cite{ES} to arbitrary graphs definable in \emph{distal} structures. Distal structures constitute a subclass of purely unstable NIP structures \cite{simon2013distal} that contains all $o$-minimal structures, various expansions of the field $\mathbb{Q}_p$ and the valued differential field of transseries (we refer to the introduction of \cite{distal} or \cite{DistValFields}  for a general discussion of distality and references). It is demonstrated in \cite{distal, chernikov2016cutting} that many of the results in semialgebraic incidence combinatorics generalize to relations definable in distal structures.

We recall some of the notions and results from \cite{chernikov2016cutting}, and refer to that article for further details. The following definition captures a combinatorial ``shadow'' of the existence of a nice topological cell decomposition (as e.g.~in $o$-minimal theories or in the $p$-adics).
 
 \begin{defn}\label{def: distal cell decomp}\cite[Section 2]{chernikov2016cutting} Let $X,Y$ be infinite sets, and $E\subseteq X\times Y$ a binary relation.
 \begin{enumerate}
 	\item Let $A\subseteq X$. For $b\in Y$, we say that $E_{b}$ \emph{crosses} $A$ if $E_{b}\cap A\neq\emptyset$ and $\left(X\setminus E_{b}\right)\cap A\neq\emptyset$.
 	\item A set $A\subseteq X$ is \emph{$E$-complete over $B\subseteq Y$} if $A$ is not crossed by any $E_{b}$ with $b\in B$.
 	\item A family $\mathcal{F}$ of subsets of $X$ is a \emph{cell decomposition for $E$ over $B\subseteq Y$} if $X\subseteq\bigcup\mathcal{F}$ and every $A\in\mathcal{F}$ is $E$-complete over $B$.
 	\item A \emph{cell decomposition for $E$} is a map $\mathcal{T}: B \mapsto \mathcal{T}(B)$ such that for each finite $B\subseteq Y$, $\mathcal{T}\left(B\right)$ is a cell decomposition for $E$ over $B$.
 	\item A cell decomposition $\mathcal{T}$ is \emph{distal} if there exist $k\in\mathbb{N}$ and a relation $D\subseteq X\times Y^{k}$ such that for all finite $B\subseteq Y$, $\mathcal{T}\left(B\right)=\{D_{\left(b_{1},\ldots,b_{k}\right)}:b_{1},\ldots,b_{k}\in B\mbox{ and }D_{\left(b_{1},\ldots,b_{k}\right)}\mbox{ is }E\mbox{-complete over }B\}$.
 	\item The \emph{exponent} of a cell decomposition $\mathcal{T}$ is the smallest $t \in \mathbb{N}$ for which there exists some $c \in \mathbb{R}_{>0}$ such that $\left|\mathcal{T}\left(B\right)\right| \leq c \left|B\right|^{t}$ for all finite sets $B \subseteq Y$.
 \end{enumerate}
 
 Existence of ``strong honest definitions'' established in \cite{chernikov2015externally} shows that every relation definable in a distal structure admits a distal cell decomposition (of some exponent).
 	
 \end{defn}
\begin{fact}
Assume that $E$ is definable in a distal structure. Then $E$ admits a distal cell decomposition.
Moreover, in this case the relation $D$ in Definition \ref{def: distal cell decomp}(5) is definable in $\mathcal{M}$.
	
\end{fact}

The following definition abstracts from the notion of cuttings in incidence geometry (see the introduction of \cite{chernikov2016cutting} for an extended discussion).
\begin{defn}
Let $X,Y$ be infinite sets, $E\subseteq X\times Y$ and $t \in \mathbb{N}_{>0}$.
  We say that
  $E$ \emph{admits cuttings with exponent $t$} if there is some constant $c \in \mathbb{R}_{>0}$ satisfying the following. For any $B \subseteq_n Y$ and any $r \in \mathbb{R}$ with $1 < r < n$ there are some sets $X_1, \ldots, X_s \subseteq X$ covering $X$ with $s \leq c r ^t$ and such that each $X_i$ is crossed by at most $\frac{n}{r}$ of the fibers $\{ E_b: b \in B \}$.
  \end{defn}

In the case $r > n$, an $r$-cutting is equivalent to a distal cell decomposition (sets in the covering are not crossed at all). And for $r$ varying between $1$ and $n$, $r$-cutting allows to control the trade-off between the number of cells in a covering and the number of times each cell is allowed to be crossed.

\begin{fact}(Distal cutting lemma, \cite[Theorem 3.2]{chernikov2016cutting})\label{fac: distal implies cutting}
Assume $E\subseteq X\times Y$ admits a distal cell decomposition $\mathcal{T}$ of exponent $t \in \mathbb{N}$. Then $E$ admits cuttings with exponent $t$.
\end{fact}

The following classical fact generalizes the Cauchy-Schwarz bound discussed at the beginning of the section.
\begin{fact}\cite{kovari1954problem}\label{fac: Kovari} Assume $E \subseteq A\times B$ is $K_{d,s}$-free, for some $d,s \in \mathbb{N}$ and $A,B$ finite. Then $|E \cap A \times B| \leq s^{\frac{1}{d}} |A||B|^{1 - \frac{1}{d}} + d |B|$.
\end{fact}
Cuttings allow to establish super-Cauchy-Schwarz estimates on the number of edges. Here we will need the following theorem, which is a slight generalization of \cite[Theorem 5.7]{chernikov2016cutting}.

\begin{thm}\label{thm: Zar gen ass}
	Assume that $E\subseteq X\times Y$ is $K_{k,k}$-free for some $k \in \mathbb{N}$, admits cuttings $\mathcal{T}$ with exponent $t \in \mathbb{N}_{\geq 2}$, and satisfies the following with $d \in \mathbb{N}_{ \geq 2}$:  there exists some $\alpha_1 \in \mathbb{R}_{>0}$  such that 
	\begin{equation}\tag{$*$}\label{eq: Zar gen ass}
	|E \cap A \times B| \leq \alpha_1 \left(|A||B|^{1-\frac{1}{d}}+|B| \right) \textrm{ for all finite } A\subseteq X, B \subseteq Y.
	\end{equation}
	Then there exists some $\alpha = \alpha(\alpha_1, k, \mathcal{T}, d) \in \mathbb{R}_{>0}$ such that 
	$$|E \cap A \times B| \leq \alpha \left(|A|^{\frac{(t-1)d}{td-1}}|B|^{\frac{td-t}{td-1}}+|A| + |B| \right)$$
	 for all finite set $A\subseteq U, B \subseteq V$. Moreover, for fixed $k, \mathcal{T},d$, $\alpha(\alpha_1,k,\mathcal{T},d)$ is a sub-linear function in $\alpha_1$.
\end{thm}
Note that for both powers we clearly have $0 < \frac{(t-1)d}{td-1}, \frac{td-t}{td-1} < 1$.
\begin{rem}\label{rem : when star holds}
The assumption \eqref{eq: Zar gen ass} is satisfied in the following cases:
\begin{enumerate}
	\item The family of sets $\mathcal{F} = \{ E_b : b \in  Y\}$ has \emph{VC-density} at most $d$ and $E$ is $K_{k,k}$-free for some $k \in \mathbb{N}$. 
	
	Then \eqref{eq: Zar gen ass} holds for some $\alpha_1>0$ by \cite[Theorem 2.1]{fox2017semi},  and Theorem \ref{thm: Zar gen ass} specializes to \cite[Theorem 5.7]{chernikov2016cutting} (see \cite[Section 5.1]{chernikov2016cutting} for a discussion of VC-density).
	\item $E$ is $K_{d,s}$-free for some $s \in \mathbb{N}$.
	
	Indeed, in this case, by Fact \ref{fac: Kovari}, for all $A,B$ we have  $|E \cap A \times B| \leq \alpha_1(|A||B|^{1 - \frac{1}{d}} + |B|)$, with $\alpha_1 = \alpha_1(d,s) := s^{\frac{1}{d}} + d$.
	
	This case of Theorem \ref{thm: Zar gen ass} generalizes/strengthens \cite[Theorem 9]{ES} (more precisely, its dual), removing the logarithmic factor. 
	
	It follows that for fixed $t, \mathcal{T}$ and $d$, $\alpha$ depends sub-linearly on $s$.
	\end{enumerate}
\end{rem}

\begin{proof}[Proof of Theorem \ref{thm: Zar gen ass}]
	The result follows from the proof of \cite[Theorem 5.7]{chernikov2016cutting}. The assumption on the VC-density in the statement of \cite[Theorem 5.7]{chernikov2016cutting} is only used to conclude that  \eqref{eq: Zar gen ass} holds with some $\alpha_1 > 0$, by Remark \ref{rem : when star holds}(1).
	Similarly, the dual version of \eqref{eq: Zar gen ass} holds with some $\alpha_4 = \alpha_4(k,\mathcal{T}) > 0$:
	$$\lvert E \cap A \times B \rvert \leq \alpha_4 \left(|B||A|^{1-\frac{1}{t}} + |A| \right) \text{ for all $A \subseteq U, B \subseteq V$}.$$
	
	As explained in the proof there, this is true as $E$ admits cuttings of exponent $t$, hence has dual VC-density $\leq t$ (by \cite[Remark 5.5]{chernikov2016cutting}), so Remark \ref{rem : when star holds}(1) applies to $E$ with the roles of the variables exchanged.
	
	 The rest of the proof only relies on \eqref{eq: Zar gen ass}, dual \eqref{eq: Zar gen ass} and $K_{k,k}$-freeness. The fact that $\alpha(\alpha_1,k, \mathcal{T},d)$ is a linear function in $\alpha_1$ for fixed $k,\mathcal{T},d$ follows by a careful inspection of the proof.
\end{proof}

\begin{cor}
\label{cor: distal incidence bound}Assume that $E \subseteq X \times Y$ admits a distal cell decomposition $\mathcal{T}$ with exponent $t$ for some $t \in \mathbb{N}_{\geq 2}$ and $E$ is $K_{2,s}$-free for some $s \in \mathbb{N}_{\geq 2}$. Then there is some $\delta=\delta\left(t \right)>0$ and $c = c(\mathcal{T}) >0$ such that for
all $A\subseteq_n X,B\subseteq_n Y$
we have $\left|E\left(A,B\right)\right| \leq cs n^{\frac{3}{2}-\delta} $.

 We can take $\delta := \frac{1}{2(2t-1)}$.
\end{cor}

\begin{proof}
By Remark \ref{rem : when star holds}(2),  \eqref{eq: Zar gen ass} holds with $d=2$, and by Theorem \ref{thm: Zar gen ass} with $d=2$ we get 
that $E(A,B) \leq \alpha \left( n^{\frac{2t-2}{2t-1}} n^{\frac{t}{2t-1}} + 2n \right) \leq (\alpha+2)\left(n^{\frac{3}{2} - \delta} \right)$ with $\delta = \frac{1}{2(2t-1)} > 0$. And, by Remark \ref{rem : when star holds}(2), $\alpha$ depends sublinearly on $s$, that is $\alpha \leq s c$ for some $c = c(\mathcal{T})$.
\end{proof}

This corollary can be interpreted as follows in an incidence-like manner. Assume that, working in a distal structure $\mathcal{M}$, we are given definable subsets $X,Y$ of some powers of $\mathcal{M}$, and a definable family $(E_a)_{a \in X}$ of subsets of $Y$ parametrized by $X$, such that $E_a \cap E_{a'}$ has bounded finite size for every $a\neq a' \in X$ (such definable families are often called \emph{normal} in model theory, e.g.~ the family of lines in $\mathbb{C}^2$ is a normal family of curves since any two distinct lines can have at most one point in common). Then Corollary \ref{cor: distal incidence bound} says that given $n$ sets in the family and $n$ elements in $Y$, the number of ``incidences'' between them is at most $O(n^{\frac{3}{2} - \delta})$ for some $\delta > 0$ --- strictly better than the bound $O(n^{\frac{3}{2}})$ given by Cauchy-Schwarz, as explained in the beginning of the section.

In particular, this provides a version of the theorem of T{\'o}th \cite{toth2015szemeredi} with a weaker bound.

\subsection{Local stability}
For the rest of Section \ref{sec: cartesian relations} we assume that $\CM = (M, \ldots)$ is a sufficiently saturated structure, $\tY, \tZ$ are definable subsets, and that $\Phi \subseteq \tY \times \tZ$ is a stable relation. 

As usual, by a $\Phi$-definable set we mean a subset $B\subseteq
\tY$  that is a finite Boolean combination of  sets defined by
$\Phi(y,c)$, $c\in \tZ$.  We write $\Phi^* \subseteq \tZ \times \tY$ for the relation obtained from $\Phi$ by exchanging the roles of the variables. Similarly we have a notion of $\Phi^*$-definable subsets of $\tZ$.
We denote by $S_\Phi(M)$ the set of all complete
$\Phi$-types on $\tY$  over $M$ (equivalently, the set of all ultrafilters on the Boolean algebra of all $\Phi$-definable subsets of $\tY$),
and similarly we denote by  $S_{\Phi^*}(M)$ the set of all complete
$\Phi^*$-types on $\tZ$. If $\UU$ is an elementary extension of $\CM$, then for
an $M$-definable set $V$ we will denote by $V(\UU)$  the set of
elements of $\UU$ realizing a formula defining $V$. 
 We say that  a $\Phi$-type $p(y)$ is \emph{non-algebraic} if in some elementary extension of $\CM$ it has a realization  outside of $M$.

The following  are  some basic facts from local stability, all of which can be found in e.g. \cite[Chapter 1, Sections 1--3]{pillay1996geometric}.
\begin{fact}\label{fact-loc-tab-def}
\begin{enumerate}
	\item For  any $p(y) \in S_\Phi(M)$, the set $\{ c\in \tZ \colon \Phi(y,c)
\in p\}$ is definable by some $\Phi^*$-formula $d_p^{\Phi}(z)$. 

Moreover, this definition is \emph{uniform}, meaning that there is some $n \in \mathbb{N}$ and a formula $d^{\Phi}(\bar{y}, z)$, $\bar{y} = (y_1, \ldots, y_n)$,  given by a finite Boolean combination of formulas $\{ \Phi(y_i,z) : i \in [n]\}$, such that for every $p \in S_{\Phi}(M)$ we have that $d_p^{\Phi}(z)$ is of the form $d^{\Phi}(\bar{c},z)$ for some $\bar{c} \in ( \tilde{Y})^n$.
\item 
Similarly, for $q(z) \in S_{\Phi^*}(M)$, the set 
$\{ b\in \tY \colon \Phi(b,z)
\in q\}$ is uniformly $\Phi$-definable via $d_q^{\Phi^*}(y)$.
\end{enumerate}

\end{fact}

\begin{fact}\label{fact:phi-ind}
Let $\UU$ be an elementary  extension of $\CM$,  
$\beta\in \tY(\UU)$ and $\gamma\in \tZ(\UU)$. 
Then $\tp_\Phi(\beta/M\gamma)$ is finitely satisfiable in $M$ if and
only if  $\tp_{\Phi^*}(\gamma/M\beta)$ is finitely satisfiable in
$M$. 
\end{fact}

\begin{defn}
  For  an elementary  extension $\UU$ of $\CM$,  
$\beta\in \tY(\UU)$ and $\gamma\in \tZ(\UU)$ we say that $\beta$ and
$\gamma$ are \emph{$\Phi$-independent over $\CM$} if 
$\tp_\Phi(\beta/M\gamma)$ is finitely satisfiable in $M$. 
\end{defn}

The following is a consequence of the fundamental  theorem of local
stability (it was also used in e.g.~\cite{hrushovski2012stable}).

\begin{fact}
  \label{fact:stable-sym} 
For types $p(y)\in S_\Phi(M)$, $q(z)\in S_{\Phi^*}(M)$ the
following conditions are equivalent. 
\begin{enumerate}
\item  There are  realizations $\beta\models p(y)$
and   $\gamma\models q(z)$  that are $\Phi$-independent over $M$
and such that $\models
  \Phi(\beta,\gamma)$;
\item  for any   realizations $\beta\models p(y)$
 and $\gamma\models q(z)$ that are $\Phi$-independent over $M$ we have  $\models
  \Phi(\beta,\gamma)$;
\item $d^\Phi_p(z)\in q(z)$;
\item  $d^{\Phi^*}_q(y)\in p(y)$.
\end{enumerate}
\end{fact}

For types $p(y)\in S_\Phi(M)$ and $q(z) \in S_{\Phi^*}(M)$
we write $\Phi(p,q)$ if one of the equivalent conditions of Fact~\ref{fact:stable-sym}
holds.

\subsection{Cartesian relations and popular types}
\label{sec:cart-relat-popul}

\begin{prop}\label{prop:cart-types} 
The relation  $\Phi$ is cartesian  if and only if  
$\models \Phi(p,q)$ for some non-algebraic types $p(y)\in S_\Phi(M)$ and $q(z)\in
S_{\Phi^*}(M)$.  
\end{prop}
\begin{proof}
Assume that $\Phi$ is cartesian, and let $B\subseteq \tY$, $C\subseteq \tZ$ be infinite sets
with $B\times C \subseteq\Phi$.   By compactness,
there is a non-algebraic type $p(y)\in S_\Phi(M)$ with
$\Phi(y,c)\in p$ for all $c\in C$.  
Hence the set $d^\Phi_p(z)$ is infinite, and we can take $q(z)$ to
be any non-algebraic type containing this formula. 

For the converse, assume that $\models \Phi(p,q)$ holds for some non-algebraic $p,q$. Then we can choose inductively a sequence $(\alpha_i, \beta_i)_{i \in \mathbb{N}}$ in $\mathcal{M}$ such that the following hold for all $i \in \mathbb{N}$:
\begin{enumerate}
	\item $\alpha_i \models d^{\Phi^*}_q(y), \beta_i \models d^\Phi_p(z)$;
	\item $\alpha_i \in \tY \setminus \{ \alpha_j : j < i \}, \beta_i \in \tZ \setminus \{ \beta_j : j < i \}$;
	\item $\models \Phi(\alpha_j,\beta_j)$ for all $j \leq i$.
\end{enumerate} 
\noindent
Indeed, assume $(\alpha_j, \beta_j : j < i)$ satisfying (1)--(3) were already chosen. 
By Fact \ref{fact:stable-sym}(4),  $d^{\Phi^*}_q(y) \in p$, and $\{ \Phi(y, \beta_j) : j <i \}  \subseteq p$ as $\beta_j \models d^{\Phi}_p(z)$ for $j<i$ by the inductive assumption. As $p$ is non-algebraic, the set
$$\{ d^{\Phi^*}_q(y)\} \cup \{ \Phi(y, \beta_j) : j <i \} \cup \{y \neq \alpha_j : j < i \}$$
of formulas with parameters in $M$ is consistent, hence realized in $\mathcal{M}$ by some $\alpha_i \in \tilde{Y}$. Similarly, $d^\Phi_p(z)\in q$ by Fact \ref{fact:stable-sym}(3), and $\{ \Phi(\alpha_j, z) : j \leq i \} \subseteq q$ as $\alpha_j \models d^{\Phi^*}_q(y)$ for all $j \leq i$ by the inductive assumption and choice of $\alpha_i$. As $q$ is non-algebraic,
$$\{ d^{\Phi}_p(z)\} \cup \{ \Phi(\alpha_j, z) : j \leq i \} \cup \{z \neq \beta_j : j < i \}$$
is consistent, and realized by some $\beta_i \in \tZ$. Then $(\alpha_j, \beta_j : j \leq i)$ satisfy (1)--(3) by construction. 

\noindent
Finally, $\{ \alpha_i : i \in \mathbb N \} \times \{\beta_j : j \in \mathbb{N} \} \subseteq \Phi$ witnesses that $\Phi$ is cartesian.

\end{proof}

The following definition is inspired by \cite{raz}.
\begin{defn}\label{def: pop type}
A non-algebraic type $p(y)\in S_\Phi(M)$ is called \emph{popular}
if the  
  set  $\{c\in \tZ \colon \Phi(y;c) \in p(y)\}$ is infinite. 

Similarly, a non-algebraic type $q(z)\in S_{\Phi^*}(M)$ is popular
if the   set  $\{b\in \tY \colon \Phi(b;z) \in q(z)\}$ is infinite. 
\end{defn}

\begin{lem}\label{lem:pop=eq} 
  \begin{enumerate}
  \item \label{item: pop types 1}A non-algebraic type $p(y)\in S_\Phi(M)$ is  popular if 
    and only if there is a non-algebraic type  $q(z)\in
    S_{\Phi^*}(M)$ with $\models\Phi(p,q)$. 
\item A non-algebraic type $q(z)\in S_{\Phi^*}(M)$ is  popular if
    and only if there is a non-algebraic type  $p(y)\in
    S_\Phi(M)$ with $\models\Phi(p,q)$. 
  \end{enumerate}
\end{lem}
\begin{proof} 
   Assume that $p(y)$ is popular, then the  definable set $d^\Phi_p(z)$ is
infinite and we can take $q$ to be any non-algebraic type containing
this set. 

Assume $\models\Phi(p,q)$ for some non-algebraic  $q(z)\in
    S_{\Phi^*}(M)$. Since $q(z)$ contains $d^\Phi_p(z)$, the  set
    defined by $d^\Phi_p(z)$ must be infinite. 
\end{proof}
Combining, we have the following equivalence.

\begin{prop}\label{prop:basic-eq}
The following conditions are equivalent. 
\begin{enumerate}
\item  The relation $\Phi$ is cartesian. 
\item There is a popular type $p(y)\in S_\Phi(M)$.
\item  There is a popular  type $q(z)\in S_{\Phi^*}(M)$. 
\end{enumerate}
\end{prop}

\subsection{Bounds on the number of edges in non-cartesian relations}

The following theorem  is a  generalization
of Theorem 1.3 in \cite{zippel}.
\begin{thm}\label{thm:zippel} 
Let $\CM$ be a sufficiently saturated structure eliminating $\exists^\infty$, and let $\tY$ and $\tZ$ be definable
  sets in $\CM$. Let $\Phi\subseteq
  \tY\times \tZ$ be an $\CM$-definable set such
  that for every $\beta\in\tilde Y$ the
  fiber $\Phi_\beta=\{ z\in \tilde Z \colon (\beta,z)\in \Phi \}$ has
  Morley rank at most $1$. Then the following conditions are equivalent. 
  \begin{enumerate}
\item The relation $\Phi$ is not
cartesian.
\item The relation $\Phi$ is $K_{k,k}$-free for some $k \in \mathbb{N}$. 
\item For all $B\subseteq_n \tY, C\subseteq_n \tZ$ we have 
\[  |\Phi \cap (B\times C)| = O(n^{3/2}).  \]
\end{enumerate}
If, in addition,  $\Phi$ admits cuttings (with some exponent $t \in \mathbb{N}$), then we also have
\begin{enumerate} 
\item[(4)] There is some $\delta>0$ such that  for all 
$B\subseteq_n \tY, C\subseteq_n \tZ$  we have 
\[  |\varPhi \cap (B\times C)| = O(n^{\frac{3}{2}-\delta}).  \]
\end{enumerate}
\end{thm}
\begin{proof}
 (1) implies (3). Assume that $\Phi$ is not cartesian.
 
 Assume first
 that there is some $b \in \tY$ for which there are some pairwise
distinct $\left(b_{i}:i\in\mathbb{N}\right)$ in $\tY$
such that $\RM\left(\Phi_b \cap \Phi_{b_i} \right) \geq1$ (hence $=1$) 
for all $i\in\mathbb{N}$. Then each of these sets contains a complete $\Phi^*$-type
of Morley rank $ 1$. By the definition of Morley rank, there are only finitely
many complete $\Phi^*$-types $q_{1},\ldots,q_{s}\in S_{\Phi^*}\left( M \right)$
with $\Phi_b\in q_i$ and $\RM \left(q_{i}\right) = 1$. But then one of these types must
contain $\Phi_{b_i}$ for infinitely many different $i$, hence
it is a popular type \textemdash{} contradicting the assumption by Proposition \ref{prop:basic-eq}.

Thus there is no $b \in \tY$ as above. Using that $T$ eliminates $\exists^\infty$, there is some $r \in \mathbb{N}$ such that for every $b \in \tY$, there are at most $r$ many $b' \in \tY$ such that $\Phi_b \cap \Phi_{b'}$ is infinite. 

Given a finite set $B \subseteq_n \tY$, consider the graph with the vertex set $B$ and the edge relation $E$ defined by $bEb'\iff \Phi_b \cap \Phi_{b'}$
is infinite. Then the graph $\left(B,E\right)$ has degree at most
$r$ by the previous paragraph, and so it is $r+1$ colorable by a standard result in graph theory. Let $B_{i} \subseteq B$
be the set of vertices corresponding to the $i$th color. Then $B = \bigsqcup_{1\leq i \leq r+1} B_i$ and, by elimination of $\exists^\infty$ again, there is some $s \in \mathbb{N}$ depending only of $\Phi$  such that $|\Phi_b \cap \Phi_{b'}| \leq s$
for any $b,b'\in B_{i}$ and $1\leq i\leq r+1$.

Now, given a finite set $C \subseteq_n \tZ$, we have that $\Phi \restriction (B_i \times C)$ is $K_{2,s}$-free for each $1 \leq i \leq r+1$. Then, using Fact \ref{fac: Kovari} with $d=2$, we have
 $$\left|\Phi \cap (B \times C) \right| \leq \sum_{i=1}^{r+1} \left|\Phi \cap (B_{i} \times C) \right|\leq (r+1)cn^{\frac{3}{2}}.$$
Hence taking $c' := (r+1)c$ depending only on $\Phi$ does the job. 
When $\Phi$ admits cuttings, we use Corollary \ref{cor: distal incidence bound} instead of Fact \ref{fac: Kovari}.

Finally, (3) implies (2) and (2) implies (1) are straightforward.
\end{proof}

\begin{rem}
Theorem \ref{thm:zippel}(4) is the only place where the assumption of the existence of a distal expansion is used. It is necessary to get a bound strictly less than $n^{\frac{3}{2}}$, as the points-lines incidence relation on the plane in an algebraically closed field of characteristic $p$ (a strongly minimal structure) demonstrates. However, this $\delta >0$ improvement is crucial to obtain the $2-\varepsilon$ bound on the power in Theorem \ref{thm:main0}(a).
\end{rem}

\section{Reducing Main Theorem to a dichotomy for binary relations}
\label{sec:main-theorem}

To prove the main theorem we 
introduce some notions and make some reductions. 
Since we are only interested in
definable subsets of products of strongly minimal sets, we may and
will assume that $\CM$ has finite Morley rank and eliminates the
quantifier $\exists^\infty $. In particular, Morley rank is \emph{additive} --- this will be used freely throughout the proof.

\begin{ass}  
For the rest of this section  (Section~\ref{sec:main-theorem}) we assume that 
 $\CM$ is a sufficiently saturated structure of finite Morley rank 
that eliminates the
quantifier $\exists^\infty$.

We fix  strongly minimal sets $X,Y,Z$  definable in $\CM$.

We also fix an $M$-definable set $F \subseteq X\times Y\times Z$  of Morley rank $2$.  

We assume that $X,Y,Z$ and $F$ are definable over the empty set (naming some constants if necessary). 
\end{ass}

Notice that writing $F$ as a  union  $F=\cup_{i=1}^k F_i$ and
applying Theorem~\ref{thm:main0} to each $F_i$, it is sufficient to
consider only the case 
when $F$ has Morley degree $1$. 

\medskip
\begin{ass}  In addition, for the rest of Section~\ref{sec:main-theorem}, we assume that $F$ has
  Morley degree 1.
\end{ass}

\subsection{Fiber-algebraic relations}
\label{sec:everywh-non-degen}\leavevmode

\begin{defn}\label{dfn:edg}
Let $S_1,S_2,S_3$ be sets and $E \subseteq S_1\times S_2 \times S_3$ be
a subset. We say that  $E$ is \emph{fiber-algebraic}  if
there is some $d\in \NN$ such that for
$\{i,j,k\}=\{1,2,3\}$ we have 
\[ \models \forall y_i \forall y_j \exists^{ \leq d} y_k E(y_1,y_2,y_3). \]  
\end{defn}

Assume, in addition to Assumptions 1 and 2, that $F\subseteq X\times
Y\times Z$ is not  cylindrical.
Then, since $Z$ is strongly minimal,  the set $\{ (a,b) \in
X\times Y \colon  \exists^{\infty}z  F(a,b,z)\}$ is finite (otherwise we can find infinite sequences $a_i \in X, b_i \in Y$ such that  $F(a_i,b_i,Z)$ is cofinite, hence $\bigcap_{i\in \mathbb{N}} F(a_i,b_i,Z)$ is infinite using saturation of $\CM$ --- so $F$ is cylindrical).  Thus
there are  co-finite  $X_0 \subseteq X$, $Y_0\subseteq Y$ such  that 
\[ \models \forall x{\in} X_0 \,\forall y{\in} Y_0\, \exists^{<\infty}z  F(x,y,z).\]

Applying the same argument for every partition of the coordinates of $F$ we conclude that if $F$ is not cylindrical then there are co-finite $X_0\subseteq X$,
$Y_0\subseteq Y$,  $Z_0 \subseteq Z$ such that the restriction of $F$
to $X_0\times Y_0 \times Z_0$ is fiber-algebraic.

It is not hard to see that if the relation $F \restriction_{X_0 \times Y_0 \times Z_0} \subseteq X_0 \times Y_0 \times Z_0$ satisfies one of the clauses $(a)$ or $(b)$ in Theorem~\ref{thm:main0}, then $F$ satisfies the same clause.
 For example, assume $F$ restricted to $X_0 \times Y_0 \times Z_0$ satisfies clause $(a)$. Let $X_1 := X \setminus X_0, Y_1 := Y \setminus Y_0, Z_1 := Z \setminus Z_0$, we have 
$$F = \bigsqcup_{i,j,k \in \{0,1\}} F \cap (X_i \times Y_j \times Z_k).$$
Then to show that $|F \cap A \times B \times C| = O(n^{2-\varepsilon})$ holds for all $A \subseteq_n X, B \subseteq_n Y, C \subseteq_n Z$, it is enough to establish the same bound for all $A \subseteq_n X_i, B \subseteq_n Y_j, C \subseteq_n Z_k$ for each choice of $i,j,k \in \{0,1\}$ separately. The case $(i,j,k) = (0,0,0)$ holds by assumption. If at least two of $i,j,k$ are $1$,  i.e.~ at least two of the sets $X_i, Y_j, Z_k$ are finite, then the bound $O(n)$ obviously holds. So we only need to consider the case when only one of $i,j,k$ is $1$, for example $F \cap X_1 \times Y_0 \times Z_0$. For any fixed $x \in X_1$ and $y \in Y_0$ there are only finitely many $z \in Z_0$ with $(x,y,z) \in F$, since $\{(a,b) \in X \times Y : \exists^\infty z F(a,b,z) \} \subseteq X_1 \times Y_1$. Hence again the bound $O(n)$ holds for $F \cap \left( X_1 \times Y_0 \times Z_0 \right)$ by elimination of $\exists^\infty$.

In view of this observation, Theorem~\ref{thm:main0} follows from
the following theorem.

\begin{thm}\label{thm:main1}
Assume, in addition to Assumptions 1 and 2, that 
  $F$ is fiber-algebraic and also that $\CM$ is interpretable in a distal
structure. Then one of the following holds. 
\begin{enumerate}[(a)]
\item  There is $\varepsilon>0$ such that for all $A\subseteq_n X ,B
  \subseteq_n Y ,C \subseteq_n Z$
  we have 
\[|F\cap A\times B \times C | = O(n^{2-\varepsilon}). \] 
\item  $F$ is group-like.
\end{enumerate}
\end{thm}

\begin{ass} For the rest of Section~\ref{sec:main-theorem}  we assume in
  addition that the relation 
 $F$ is fiber-algebraic and  
$d\in \NN$ is as in Definition~\ref{dfn:edg}.
\end{ass}

\subsection{On acl-diagrams}
\label{sec:acl-diagrams}

We say that three elements $p_1,p_2,p_3$ of $M$  form an
\emph{$\acl$-triangle} if  $\RM(p_i/\emptyset)=1$ for $i\in[3]$,
$\RM(p_1p_2p_3/\emptyset)=2$,  and
for all $\{i,j,k\}=\{1,2,3\}$ we have $p_i\in \acl(p_jp_k)$ (hence also $p_i \ind p_j$ for all $i\neq j \in \{1,2,3\})$.

Since $F$ is fiber-algebraic  of Morley rank $2$, we have the following
claim.

\begin{claim}\label{claim:2-ind} Let $(a,b,c)$  be
  generic in $F$, i.e.~$(a,b,c)\in F$ and $\RM(abc/\emptyset)=2$. 
Then $a,b,c$ form an $\acl$-triangle.

In particular, $b$ and $c$ are independent generics
  in $Y$ and $Z$, respectively; and, by stationarity (as $Y\times Z$ has Morley degree $1$),  for any independent generics
  $b'\in Y$, $c'\in Z$ there is some $x\in X$ such that
  $(x,b',c') \in F$ and $(x,b',c')$ is generic in $F$).

  Similarly,  for any independent generics
  $a'\in X$, $b'\in Y$ we have
  $\CM\models \exists z \, F(a',b',z)$.
\end{claim}

\medskip
In this paper we will consider some simple diagrams, where by 
a \emph{diagram} we mean a collection of elements of $M$   and lines between
them (subsets of the given elements). 

\begin{defn}
We say that a given diagram is an \emph{$\acl$-diagram} if 
\begin{enumerate}
\item $\RM(p/\emptyset)=1$ for every  point $p$ in the diagram;
\item  Every three collinear points form an $\acl$-triangle;
\item $\RM(p q r/\emptyset)=3$ for any three  non-colinear
  points $p,q,r$.
\end{enumerate}
\end{defn}

\subsection{The $4$-ary relation $Q$}
\label{sec:relation-w}

Our next goal is to restate Theorem~\ref{thm:main1} in term of a $4$-ary
relation $Q$.  (We continue to use  Assumptions 1--3.)
\medskip

Let $Q\subseteq Y^2\times Z^2$ be the
definable relation
\begin{equation*}
   Q=\{ (y,y',z,z')\in  Y^2\times Z^2  \colon \exists\,x{\in} X
  \left((x,y,z){\in F} \,\&\,
  (x,y',z'){\in} F \right) \}.
\end{equation*}

We first observe some basic properties of $Q$.

\begin{claim}\label{claim:d-bound}  For $(b_1,b_2)\in Y^2$ and
  $c_1\in Z$   the set
\[ \{ z\in Z \colon (b_1,b_2, c_1,z) \in Q \} \]
has size at most $d^2$.   

Similarly,  for $(c_1,c_2)\in Z^2$ and
  $b_1\in Y$   the set
\[ \{ y\in Y \colon (b_1, y , c_1,c_2) \in Q \} \]
has size at most $d^2$.   
\end{claim}
\begin{proof}
Let $b_1, b_2, c_1$ be fixed. As $F$ is 
fiber-algebraic, by the choice of $d$ there are at most $d$ elements $x \in X$ such that $(x,b_1,c_1) \in F$, and for each such $x$, there are at most $d$ elements $z \in Z$ such that $(x,b_2,z) \in F$. Hence, by definition of $Q$, there are at most $d^2$ elements $z \in Z$ such that $(b_1,b_2,c_1,z) \in Q$.    
\end{proof}

\begin{cor}\label{cor:dim-fiber-i-one}
For any $\bar b=(b_1,b_2)\in Y^2$ the Morley rank of the fiber
$Q_{\bar b}=\{ \bar z\in Z^2 \colon (\bar b,\bar z)\in Q\}$ is at
most $1$.

Similarly,
for any $\bar c=(c_1,c_2)\in Z^2$  the Morley rank of the fiber
$Q_{\bar c}=\{ \bar y\in Y^2 \colon (\bar y,\bar c)\in Q\}$ is at
most $1$.

\end{cor}

\begin{cor}\label{cor:w-mr3}
  The Morley rank of $Q$ is $3$. 
\end{cor}
\begin{proof} It follows from the additivity of Morley rank and Corollary \ref{cor:dim-fiber-i-one} that $\RM(Q)\leq \RM(Y^2) + 1 = 3$.

On the other hand, let $b_1,b_2\in Y, c_1\in Z$ be
independent generics, hence $\RM(b_1b_2c_1/\emptyset)=3$.  
By Claim~\ref{claim:2-ind}, there is a generic $a\in
X$ with $\CM\models F(a,b_1,c_1)$. Applying
Claim~\ref{claim:2-ind} to $a$ and $b_2$ (as $a \in \acl(b_1 c_1)$ and $b_1 c_1 \ind b_2$, we have $a\ind b_2$)  we get
$c_2\in Z$ with $\CM\models F(a,b_2,c_2)$.
Since $\CM\models Q(b_1,b_2,c_1,c
_2)$ we have
$\RM(Q)\geq 3$.

\end{proof}

\begin{lem}\label{lem:almost-g} 
Let 
$(b_1,b_2,c_1,c_2)$ be generic in $Q$. Then there is
$a\in X$ such that  the diagram 
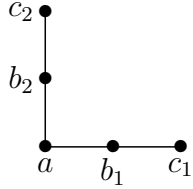
\[
\begin{tikzpicture}[scale=0.9]
\draw[black,line  width=0.2mm] (0,0) -- (0,2); 
\draw[black,line  width=0.2mm] (0,0) -- (2,0); 
%\draw[black,line  width=0.2mm] (1,0) -- (0,2); 
%\draw[black,line  width=0.2mm] (0,1) -- (2,0); 
\node at (0,0) {$\bullet$};
\node[below,scale=1.0] at (0,0) {$a$}; 
\node at (1,0) {$\bullet$}; 
\node[below,scale=1.0] at (1,0) {$b_1$}; 
\node at (2,0) {$\bullet$}; 
\node[below,scale=1.0] at (2,0) {$c_1$}; 
\node at (0,1) {$\bullet$}; 
\node[left,scale=1.0] at (0,1) {$b_2$}; 
\node at (0,2) {$\bullet$}; 
\node[left,scale=1.0] at (0,2) {$c_2$}; 
%\node at (0.65,0.65) {$\bullet$}; 
%\node at (0.9,0.9) {$t$}; 
\end{tikzpicture}
\]
is an $\acl$-diagram with $\CM\models F(a,b_i,c_i)$. 
\end{lem}
\begin{proof}
By the definition of $Q$ we can find $a\in X$ with 
\[ \CM\models F(a,b_1,c_1) \,\&\, F(a,b_2,c_2). \] 
We claim that for this $a$ the diagram above is an $\acl$-diagram. 
  
First notice that $\RM(p/\emptyset)\leq 1$ for any point $p$.

Secondly,  by fiber-algebraicity of $F$,
if 
$p,q,r$ are three 
non-colinear points  then every point of the diagram is in 
$\acl(pqr)$;   in particular, since $\RM(b_1b_2c_1c_2/\emptyset)=3$, we have 
 $\RM(pqr/\emptyset)=3$.

Finally, given any two distinct non-collinear points $(p,q)$ we can find a point
$r$ such that $p,q,r$ are non-collinear (so $\RM(pqr/\emptyset) = 3$), hence $\RM(pq/\emptyset)=2$ by additivity of Morley rank.  

It follows then that any three collinear points, after reordering,
form a generic realization of $F$, and hence, by Claim
\ref{claim:2-ind}, an $\acl$-triangle.  
\end{proof}

\subsection{The clause (a) in Theorem~\ref{thm:main1}.}
\label{sec:further-reduction}
In this section we state a property of the relation $Q$ that implies
the clause (a) in Theorem~\ref{thm:main1}.

First we need some basic counting properties.

\begin{cor}\label{cor:point-set-count}
For any $\bar b \in Y^2$ and a finite set $C\subseteq Z$ we
have 
\[ | Q\cap  (\{\bar b\} \times C^2)| \leq d^2|C|. \] 

Similarly, for any $\bar c \in Z^2$ and a finite set $B\subseteq Y$ we
have 
\[ | Q\cap (B^2\times \{\bar c\})| \leq d^2|B|. \] 
\end{cor}
\begin{proof}
  Follows from Claim~\ref{claim:d-bound}
\end{proof}

The following bound is similar to \cite[Lemma 2.2]{raz}.

\begin{prop}\label{prop:bound-1}
Let $A \subseteq X, B\subseteq  Y, C\subseteq Z$ be finite. Then for 
$F'=F\cap (A\times B \times C)$  and $Q'= Q\cap
(B^2\times C^2)$
we have 
\[ |F'| \leq d |A|^{1/2}|Q'|^{1/2}. \]
\end{prop}
\begin{proof}
  
Let $W \subseteq X\times Y^2\times Z^2$ be the definable set 
\begin{equation*}
  W=\{ (x,y,y',z,z')\in X\times Y^2\times Z^2  \colon (x,y,z)\in F\,\&\,
  (x,y'z')\in F \},
\end{equation*}
and let $W'=W\cap (A\times B^2\times C^2)$.  

As usual, for a set $S\subseteq A\times D$ and $a\in A$ we denote by
$S_a$ the fiber $S_a=\{ u\in D \colon (a,u)\in S \}.$

 Notice that $|F'|=\sum_{a\in A} |F'_a|$, and  
$|W'|=\sum_{a\in A} |F'_a|^2$. 
By the Cauchy–-Schwarz inequality,
\[ |F'| \leq |A|^{1/2}    \Bigl(\sum_{a\in A}|F'_a|^2\Bigr)^{1/2} =
|A|^{1/2}|W'|^{1/2} .\]

For a point $g\in Q'$, the fiber $W'_g$  
has size at most $d$ as $F$ is fiber-algebraic, hence 
$|W'|\leq d |Q'|$ and $|F'| \leq  d |A|^{1/2}|Q'|^{1/2}$.
\end{proof}

The next proposition shows that the bound $O(n^{3/2-\delta})$ for $Q$
translates to the bound $O(n^{2-\delta})$ for $F$.  

\begin{prop}\label{prop:three-two} Let $\tilde
  Y := Y^2$, $\tilde Z := Z^2$, and we view $Q$ as a subset of
  $\tilde Y\times \tilde Z$. Assume that there are definable sets
  $Y_0 \subseteq \tilde Y$ and $Z_0\subseteq \tilde Z$ of Morley rank
  at most $1$ such that for some $0 < \delta < \frac{1}{2}$  for all
  $B'\subseteq_m \tilde Y\setminus Y_0, C'\subseteq_m \tilde
  Z\setminus Z_0$ we have 
$ |Q\cap (A'\times B')|=O(m^{3/2-\delta})$. 
Then $F$ satisfies the clause $(a)$ of Theorem~\ref{thm:main1} (with $\varepsilon = \delta$).
 \end{prop}
 \begin{proof} We fix $Y_0 \subseteq \tY$ and $Z_0 \subseteq \tZ$  of
   Morley rank at most $1$, $c_0 \in \RR$ and $\delta> 0$ such
   that for all $m$ large enough and for all 
   $B'\subseteq_m  \tY\setminus Y_0, C'\subseteq_m
  \tZ\setminus Z_0$ we have 
$ |Q\cap A'\times B'| \leq c_0 m^{3/2-\delta}$. 

Since $Y_0$ has Morley  rank at most $1$, using
elimination of $\exists^\infty$,  it is not hard to see that there is
$k_1\in \NN$ such that for any
finite $B\subseteq Y$ we have $|B^2 \cap Y_0| \leq k_1|B|$.

Similarly, there is $k_2\in \NN$ such that for any
finite $C\subseteq Z$ we have $|C^2 \cap Z_0| \leq k_2|C|$.

Given $A \subseteq_n X$, $B\subseteq_n
Y$, $C\subseteq_n Z$, let 
$B'=B^2\cap (\tY\setminus Y_0)$ and  $C'=C^2\cap (\tZ\setminus Z_0)$. 
Obviously $|B'| \leq  n^2$ and $|C'|\leq n^2$.

We have 
\begin{gather*}
|Q\cap (B^2\times C^2)| \leq \\
|Q\cap (B'\times C')| + \left \lvert Q \cap \left( \left(B^2\cap Y_0 \right)\times C^2 \right) \right \rvert+
  \left \lvert Q \cap \left( B^2 \times \left(C^2\cap Z_0 \right) \right) \right \rvert .
\end{gather*}
By our assumptions $|Q\cap (B'\times C')|\leq  c_0(n^2)^{\frac{3}{2}-\delta} = c_0 n^{3 - 2\delta}$. 
Since $|B^2\cap Y_0| \leq k_1 n$, from
Corollary~\ref{cor:point-set-count},  we get
$\left \lvert Q \cap \left( \left(B^2\cap Y_0 \right)\times C^2 \right) \right \rvert \leq k_1 d^2 n^2$; and similarly 
 $\left \lvert Q \cap \left( B^2 \times \left(C^2\cap Z_0 \right) \right) \right \rvert \leq  k_2 d^2 n^2$. 

Thus $|Q\cap (B^2\times C^2)| \leq c_1 n^{3-2\delta}$,  where $c_1>0$ does not depend on $n, A,B,C$. 

Applying Proposition~\ref{prop:bound-1}, we obtain 
\[|F\cap (A\times
B\times C)| \leq d  (n c_1n^{3-2\delta})^{1/2}=O(n^{2-\delta}).\]
 \end{proof}

Combining this with Theorem \ref{thm:zippel},  we obtain a property of $Q$ that implies the clause $(a)$ in Theorem~\ref{thm:main1}.

\begin{prop}\label{prop:three-two} Let $\tilde
  Y =Y^2$, $\tilde Z=Z^2$, and we view $Q$ as a subset of
  $\tilde Y\times \tilde Z$. Assume in addition that $Q$  admits cuttings (with some exponent). Assume also that there are definable sets
  $Y_0 \subseteq \tilde Y$ and $Z_0\subseteq \tilde Z$ of Morley rank
  at most $1$ such that the restriction of $Q$ to $(\tY\setminus
  Y_0)\times (\tZ\setminus Z_0)$ is not cartesian.
Then $F$ satisfies the clause $(a)$ of Theorem~\ref{thm:main1}.
 \end{prop}

\subsection{The clause (b) of Theorem~\ref{thm:main1}.}

We fix a saturated elementary extension $\UU$ of $\CM$.

\begin{prop}\label{prop:w-case-b}   Assume there are
  $\beta=(\beta_1,\beta_2)\in Y^2(\UU)$ and
  $\gamma=(\gamma_1,\gamma_2)\in Z^2(\UU)$ with $(\beta,\gamma)\in Q(\UU)$, such that 
  $\RM(\beta/M), \RM(\gamma/M) >0$, 
$\beta \ind_M \gamma$
and $\acl(\beta)\cap \acl(\gamma) \not\subseteq
\acl(\emptyset)$. Then $F$  is group-like.
\end{prop}
\begin{proof}
Choose $t\in  \bigl(\acl(\beta)\cap \acl(\gamma)\bigr) \setminus
\acl(\emptyset)$.  We first list some properties of $\beta,\gamma$
and $t$. 

\begin{enumerate}[(i)]
\item   Since $\beta \ind_M \gamma$, and $t\in  \bigl(\acl(\beta)
  \cap \acl(\gamma)\bigr)$   we have $$t\in M.$$ 
\item Since $t\in \acl(\beta) \setminus \acl(\emptyset)$ we have 
\[ \beta\nind_\emptyset  t; \] 
and, similarly, \[\gamma\nind_\emptyset t.\]
\item From $(i)$ and $(ii)$,  
since $\RM(\beta/\emptyset)\leq 2$ and  $\RM(\beta/M)>0$, we
  obtain 
\[\RM(\beta/\emptyset)=2 \text{ and } \RM(\beta/t) = \RM(\beta/M)=1;\]
and, similarly, 
\[ \RM(\gamma/\emptyset)=2 \text{ and }\RM(\gamma/t) =
  \RM(\gamma/M)=1.\]
\item Since $t\in \acl(\beta) \setminus \acl(\emptyset)$ and 
$\beta\notin \acl(t)$ we have $\RM(t/\emptyset)=1$.
\item \label{item: 5 in group config} Since $\beta$ and $\gamma$ are independent over $M$ we have
  $RM(\beta\gamma/M)=2$, and since $\beta \nind_\emptyset  M$, we have 
  $RM(\beta\gamma/\emptyset)=3$, i.e. $(\beta,\gamma)$ is generic in
  $Q(\UU)$.
\item \label{item: 6 in group config} It follows from \ref{item: 5 in group config} and Lemma \ref{lem:almost-g} that $\beta_i
  \notin \acl(\gamma)$ and $\gamma_i\notin\acl(\beta)$ for $i\in[2]$. 
\end{enumerate}

We also have that both $\{ \beta_1,\beta_2,t\}$ and  $\{
\gamma_1,\gamma_2,t\}$ are $\acl$-triangles.  Indeed, for example,  since
$\RM(\beta_1\beta_2t/\emptyset)=2$, 
to show that 
 $\{ \beta_1,\beta_2,t\}$ is a triangle it is sufficient to check that
 $t\not\in \acl(\beta_i)$ for $i=1,2$. But if $t\in \acl(\beta_i)$,
 then $\beta_i\in\acl(t)$, hence $\beta_i\in \acl(\gamma)$ --- 
 contradicting \ref{item: 6 in group config}.

By Lemma~\ref{lem:almost-g} there is $\alpha\in X(\UU)$ such that the
diagram 
\begin{equation}
  \label{eq:1}
\begin{tikzpicture}[scale=0.9]
\draw[black,line  width=0.2mm] (0,0) -- (0,2); 
\draw[black,line  width=0.2mm] (0,0) -- (2,0); 
%\draw[black,line  width=0.2mm] (1,0) -- (0,2); 
%\draw[black,line  width=0.2mm] (0,1) -- (2,0); 
\node at (0,0) {$\bullet$};
\node[below,scale=1.0] at (0,0) {$\alpha$}; 
\node at (1,0) {$\bullet$}; 
\node[below,scale=1.0] at (1,0) {$\beta_1$}; 
\node at (2,0) {$\bullet$}; 
\node[below,scale=1.0] at (2,0) {$\gamma_1$}; 
\node at (0,1) {$\bullet$}; 
\node[left,scale=1.0] at (0,1) {$\gamma_2$}; 
\node at (0,2) {$\bullet$}; 
\node[left,scale=1.0] at (0,2) {$\beta_2$}; 
%\node at (0.65,0.65) {$\bullet$}; 
%\node at (0.9,0.9) {$t$}; 
\end{tikzpicture}
\end{equation}
is an $\acl$-diagram with $\UU\models F(\alpha,\beta_i,\gamma_i)$.

We  claim that 
\begin{equation}
  \label{eq:2}
  \begin{tikzpicture}[scale=0.9]
\draw[black,line  width=0.2mm] (0,0) -- (0,2); 
\draw[black,line  width=0.2mm] (0,0) -- (2,0); 
\draw[black,line  width=0.2mm] (1,0) -- (0,2); 
\draw[black,line  width=0.2mm] (0,1) -- (2,0); 
\node at (0,0) {$\bullet$};
\node[below,scale=1.0] at (0,0) {$\alpha$}; 
\node at (1,0) {$\bullet$}; 
\node[below,scale=1.0] at (1,0) {$\beta_1$}; 
\node at (2,0) {$\bullet$}; 
\node[below,scale=1.0] at (2,0) {$\gamma_1$}; 
\node at (0,1) {$\bullet$}; 
\node[left,scale=1.0] at (0,1) {$\gamma_2$}; 
\node at (0,2) {$\bullet$}; 
\node[left,scale=1.0] at (0,2) {$\beta_2$}; 
\node at (0.65,0.65) {$\bullet$}; 
\node at (0.9,0.9) {$t$}; 
\end{tikzpicture}
\end{equation}
is an $\acl$-diagram.

 We already have that every three colinear points form an
 $\acl$-triangle, and it is sufficient to check that for three
 non-colinear points $\{p,q,r\}$ we have $\RM(pqr/\emptyset)=3$. 
 If $t\notin\{p,q,r\}$  then it follows from the diagram~\eqref{eq:1}.
 Assume $t\in\{p,q,r\}$, say $\{p,q,r\}=\{t, \alpha, \beta_1\}$. Then
 $\{t,\beta_1\}$ is inter-algebraic with $\{\beta_1,\beta_2\}$ and hence
 $$\RM(t\alpha\beta_1/\emptyset)=\RM(\beta_2\alpha\beta_1/\emptyset)=3.$$
 The same argument works for any three non-collinear points containing
 $t$. 
 
 Thus the diagram~\eqref{eq:2} is an $\acl$-diagram. It follows from the Group Configuration Theorem in stable theories (see {\cite[Theorem
  6.1]{HZ}} and the discussion in \cite[Section 6.2]{HZ})
that $F$ is group-like.
\end{proof}

%\subsection{Restatement of Theorem~\ref{thm:main1} in terms of the relation $G$}
%\label{sec:rest-theor-terms}
%
%\begin{thm}\label{thm:main-3}
%Let $\tilde
%  Y =Y^2$, $\tilde Z=Z^2$, and we view $G$ as a subset of
%  $\tilde Y\times \tilde Z$. Then one of the following holds.
%  \begin{enumerate}[(a)]
%  \item There are definable sets $Y_0\subseteq \tY$ and $Z_0\supseteq
%    \tZ$ of Morley Rank at most one such that the restriction of $G$
%    to $(\tY\setminus Y_0)\times (\tZ\setminus Z_0)$ is not cartesian.
%\item  In an elementary extension $\UU$ of $\CM$ there are
%  $\beta\in\tY(\UU)$ and
%  $\gamma\in \tZ(\UU)$ with $(\beta,\gamma)\in G(\UU)$, such that 
%  $\RM(\beta/M) >0$,    $\RM(\gamma/M) >0$, 
%$\beta \ind_M \gamma$
%and $\acl^\eq(\beta)\cap \acl^\eq(\gamma) \not\subseteq
%\acl^\eq(\emptyset) $. 
%\end{enumerate}
%\end{thm}

\section{Dichotomy for binary relations}
\label{sec:proof-theor-refthm:m}

In this section we prove a dichotomy theorem for binary relations between sets of Morley rank $2$. By Propositions~\ref{prop:three-two} and \ref{prop:w-case-b},
Theorem~\ref{thm:main1} follows from Theorem \ref{thm:main-4} applied with $\Phi := Q, \tilde{Y} := Y^2, \tilde{Z} := Z^2$ as in Section \ref{sec:main-theorem}.

\begin{thm}\label{thm:main-4}
Let $\CM$ be a sufficiently saturated structure of finite Morley rank that eliminates quantifier $\exists^\infty$.

Let $\tY$ and $\tZ$ be $M$-definable sets of Morley rank $2$ and Morley
degree $1$. Let $\varPhi \subseteq \tY\times \tZ$ be a definable
subset of Morley rank $3$. Then one of the following holds.
\begin{enumerate}[(a)]
  \item \label{item: Dichotomy for G 1}There are definable sets $Y_0\subseteq \tY$ and $Z_0\subseteq
    \tZ$ of Morley rank at most $1$ such that the restriction of $\varPhi$
    to $(\tY\setminus Y_0)\times (\tZ\setminus Z_0)$ is not cartesian.
\item  \label{item: Dichotomy for G 2}There are
  $\beta\in\tY(\UU)$ and
  $\gamma\in \tZ(\UU)$ with $(\beta,\gamma)\in \varPhi(\UU)$, such that 
  $\RM(\beta/M), \RM(\gamma/M) >0$, 
$\beta \ind_M \gamma$
and $\acl(\beta)\cap \acl(\gamma) \not\subseteq
\acl(\emptyset) $. 
\end{enumerate}
\end{thm}

\begin{proof}
As usual, for a $\Phi$-type $p(y)\in
S_\Phi(M)$,  we denote by $\RM(p(y))$ the Morley rank of $p$ as
an incomplete type.

We assume that \ref{item: Dichotomy for G 1} doesn't hold, and show that then \ref{item: Dichotomy for G 2} must hold. 
For a definable set $\tY' \subseteq \tY$ we say that 
$\tY'$ is \emph{large} in  $\tY$ if $\RM(\tY \setminus \tY') \leq  1$;
and the same for a subset $\tZ' \subseteq \tZ$. Notice that in the proof we can freely replace $\tY$ and $\tZ$ by their
large subsets.

Let $p^*(y)\in S(M)$ be the generic type on $\tY$, it is the unique type
on $\tY$ of Morley rank $2$ (as $\tY$ has Morley degree $1$ by assumption).  

Since $\Phi$ has Morley rank $3$, the set $\{ c\in \tZ \colon
\Phi_c \in p^* \}$ is definable (by Fact \ref{fact-loc-tab-def}) and
has Morley rank at most $1$ by additivity of Morley rank. Thus we can
throw away this set and assume that the Morley rank of $\Phi_c$
is at most $1$ for all  $c\in \tZ$. 
Similarly, we may assume that  the Morley rank of $\Phi_b \subseteq \tZ$
is at most $1$ for all  $b\in \tY$.

Assume that $p$ is a popular type for $\Phi$ (see Definition \ref{def: pop type}). Then $\RM(p) = 1$ ($\RM(p) \geq 1$ as $p$ is non-algebraic, and $\RM(p) \leq 1$ as $\Phi_c \in p$ for some $c$ by definition of popular types, and $\RM(\Phi_c) = 1$ by the previous paragraph). If there are only finitely many popular types for $\Phi$, then we can
throw
away finitely many definable sets of Morley rank $1$ (one in each of the popular types),  and pass to a large subset on which there are no popular types (hence, obtaining \ref{item: Dichotomy for G 1} using Proposition \ref{prop:basic-eq}).  Thus we can assume that there are
infinitely many popular types on $\tY$.

\gdef\CQ{\mathcal{Q}}
Let $\CP$ be the set of all popular types on $\tY$ and $\CQ$ be the set of all popular types on $\tZ$.

For $p \in S_\Phi(M)$ we denote by $[p] \in \CM^\eq$ the canonical parameter of 
the $\Phi^*$-definable set  $d_p^\Phi = \{ c\in \tZ \colon \Phi_c \in p
\}$. Recall that by Fact \ref{fact-loc-tab-def} the definition $d_p^{\Phi}$ is given by instances of the same formula for all $p$, hence $[p]$ is the canonical parameter of an instance of the same formula for all $p$, hence $\{[p] : p \in S_\Phi(M)\}$ is a subset of a fixed sort in $\mathcal{M}^{\eq}$; similarly, for  $q\in S_{\Phi^*}(M)$ we will denote by $[q] \in \CM^\eq$ the
canonical parameter for $d^{\Phi^*}_q$. 

Clearly both  maps  $p \mapsto [p]$ and $q\mapsto [q]$ are injective.

\begin{claim}
 The sets $\{ [p] \colon p\in \CP \}$ and $\{ [q] \colon q\in \CQ \}$ are $\emptyset$-type-definable subsets of the corresponding sorts in $\mathcal{M}^{\eq}$. 
  \end{claim}

\begin{proof}
Using that $\CM$ eliminates $\exists^\infty$ (and uniform definability of types), the desired set $\{ [p] : p \in \mathcal{P}\}$  is type-definable by
$$\{[p] : \exists^\infty z d_p^\Phi(z) \land \bigwedge_{n \in \mathbb{N}} (\forall z_1 \ldots \forall z_n ( \bigwedge_{i=1}^n d^\Phi_p(z_i) \rightarrow \exists^\infty y \bigwedge_{i=1}^n\Phi(y,z_i) \}.$$
And similarly for $\mathcal{Q}$.
\end{proof}

\begin{claim} If $p\in \CP$ and $c\in d^\Phi_p$ then $[p]\in
  \acl(c)$.   
    
Similarly, if $q\in \CQ$ and $b\in d^{\Phi^*}_q$ then $[q]\in
  \acl(b)$.   
\end{claim}
\begin{proof}
As for any $c \in \tZ$ there are only finitely many $\Phi$-types of Morley rank $1$ containing $\Phi_c$.
\end{proof}

\begin{claim}\label{lem:fin} For any $p\in \CP$ there are only finitely many $q\in \CQ$
  with $\models\Phi(p,q)$, and vice versa.
\end{claim}
\begin{proof}  Let $\beta\in \UU$ realize $p$. It is not hard to see
  that, for $q\in \CQ$, if  $\models\Phi(p,q)$ then the Morley rank
  of the partial type $\{ \Phi(\beta,z) \}\cup q(z)$ is $1$.  Since  the Morley rank of
  $\Phi(\beta,z)$ is $1$, there only finitely many such $q$.  
  \end{proof}

Since the set $\CP$ is infinite,  we can find a popular type $p \in \CP$ such that $[p] \notin \acl(\varnothing)$. 
Choose $q\in \CQ$ with $\models \Phi(p,q)$.

Choose some $\beta\models p(y)$ and $\gamma\models q(z)$ independent over
$M$. 

We have $[p]\in \acl(\gamma)$, $[q]\in \acl(\beta)$ with $[p]$ and
$[q]$ inter-algebraic over the empty set.  Hence \ref{item: Dichotomy for G 2} holds.
\end{proof}

\bibliography{refs}

\end{document}

%% file: declar.tex
\newtheorem{thm}{Theorem}[section]
\newtheorem{cor}[thm]{Corollary}
\newtheorem{lem}[thm]{Lemma}
\newtheorem{prop}[thm]{Proposition}

\newtheorem{claim}[thm]{Claim}
\newtheorem{fact}[thm]{Fact}
\newtheorem{ass}{Assumption}
\theoremstyle{definition}
\newtheorem{defn}[thm]{Definition}
\theoremstyle{remark}

\newtheorem{rem}[thm]{Remark}

 \numberwithin{equation}{section}
    {\medskip\begingroup\leftskip 0.5cm\rightskip 0.5cm\noindent\begin{small}{\bf Remark.}}
    {\end{small}\par\endgroup}
{\begin{list}{$\bullet$}
 {\settowidth{\labelwidth}{\textsf{$\bullet$}} \setlength{\leftmargin}{10pt}}}
{\end{list}}

\newcounter{ssample}[section]

{\color{Bittersweet} \noindent Example \refstepcounter{ssample}\hbox{\bf \arabic{section}.\arabic{ssample}.}}
{}

\newcounter{insertcount}

    {\color{blue} \medskip\begingroup\noindent\begin{small}{\color{blue} \stepcounter{insertcount}
          {
            \bf Insert \arabic{insertcount}. #1.}
            \addcontentsline{toc}{subsection}{{\ \ \small  Insert \arabic{insertcount}: #1}}
               \leavevmode  }
           }
    {\end{small}\par\endgroup}

\newcommand{\mrmk}[1]%marginal remark
{{\tiny$^{\spadesuit}$}\marginpar{\fbox{\footnotesize #1}}}
\def\strutdepth{\dp\strutbox}%
\def\marginalnote#1{\strut\vadjust{\kern-\strutdepth\specialnote{#1}}}%
\def\specialnote#1{\vtop to \strutdepth{\baselineskip%
\strutdepth\vss\llap{\hbox{\scriptsize \bf #1}}\null}}%

\newcommand{\RR}{\mathbb{R}}
\newcommand{\CC}{\mathbb{C}}

\newcommand{\NN}{\mathbb N}

\def\tp{\mathrm{tp}}

\reversemarginpar

%% file: localdef.tex
\def\CM{\mathcal{M}}

\def\UU{\mathbb{U}}

\newcommand*\bbar[1]{%
  \hbox{%
    \vbox{%
      \hrule height 0.5pt % The actual bar
      \kern0.5ex%         % Distance between bar and symbol
      \hbox{%
        \kern-0.1em%      % Shortening on the left side
        \ensuremath{#1}%
        \kern-0.1em%      % Shortening on the right side
      }%
    }%
  }%
}

\gdef\CP{\mathcal{P}}

\gdef\eq{\mathrm{eq}}

\def\Ind#1#2{#1\setbox0=\hbox{$#1x$}\kern\wd0\hbox to 0pt{\hss$#1\mid$\hss}
\lower.9\ht0\hbox to 0pt{\hss$#1\smile$\hss}\kern\wd0}

\def\ind{\mathop{\mathpalette\Ind{}}}

\def\notind#1#2{#1\setbox0=\hbox{$#1x$}\kern\wd0
\hbox to 0pt{\mathchardef\nn=12854\hss$#1\nn$\kern1.4\wd0\hss}
\hbox to 0pt{\hss$#1\mid$\hss}\lower.9\ht0 \hbox to 0pt{\hss$#1\smile$\hss}\kern\wd0}

\def\nind{\mathop{\mathpalette\notind{}}}